\documentclass[10pt]{article}
\usepackage[utf8]{inputenc}
\usepackage[spanish]{babel}
\usepackage{lipsum}
\usepackage[nottoc]{tocbibind}
\usepackage{imakeidx} 
\indexsetup{othercode=\small}
\makeindex[program=makeindex,columns=2,options={-s index_style.ist},intoc]
\usepackage{url}
\usepackage[pdftex,hyperindex]{hyperref} 
\usepackage{amsmath,amsfonts,amstext,amssymb, pifont, stmaryrd,mathabx, relsize}
\usepackage{graphicx}
\usepackage[all,2cell]{xy}
\usepackage[toc]{appendix}
\usepackage{geometry}
\usepackage{nomencl}
\makenomenclature
\usepackage{amsthm}
\usepackage{thmtools}
\usepackage[dvipsnames]{xcolor}
\usepackage{tikz-cd}
\usepackage{tikz}
\usetikzlibrary{positioning, arrows}
\usepackage{comment}

\newtheorem{theorem}{Teorema}[section]
\newtheorem{proposition}[theorem]{Proposición}
\newtheorem{lemma}[theorem]{Lema}

\theoremstyle{definition}
\newtheorem{definition}[theorem]{Definición}
\theoremstyle{remark}
\newtheorem{remark}[theorem]{Obervación}
\theoremstyle{remark}
\newtheorem{example}[theorem]{Ejemplo}

\newcommand{\field}{F}
\newcommand{\dg}{\text{DG}}

\newcommand{\dga}{\text{DGA}}
\newcommand{\dgacat}{\text{DGA-Mod}}

\newcommand{\coeq}{\overset{\curvearrowleft}{\rightrightarrows}}

\begin{document}
\title{Colimites en la categoría de operads simétricos}
\author{Jesús Sánchez Guevara}
\date{%
    Noviembre 2022\\
    Escuela de Matemática\\%
    Universidad de Costa Rica, CR.\\
	\href{mailto: jesus.sanchez_g@ucr.ac.cr}{jesus.sanchez\_g@ucr.ac.cr}\\
}

\maketitle
	\begin{abstract}
	En el presente artículo se hace un estudio detallado de las construcciones
	existentes de colimites en la categoría de operads simétricos.
	Además, se presentan algunos ejemplos de operads obtenidos a partir
	de colímites de otros operads.
	\end{abstract}
\section{Introducción}

La noción de operad aparece en los años sesenta cuando James Dillon Stasheff 
usa $A_\infty$-álgebras para estudiar estructuras algebraicas que tienen propiedades asociativas no estrictas
(\cite{10.2307/1993608} y \cite{10.2307/1993609}), es decir, que se cumplen bajo equivalencias de homotopía.
Luego, con el fin de describir la estructura homotópica subyacente a los espacios de lazos,
Peter May introduce en los inicios de los años setenta el concepto de operad (\cite{may2007}),
el cual, entre otras cosas, engloba las álgebras estudiadas por Stasheff.
En su trabajo, May demuestra que, si un espacio tiene una estructura de álgebra sobre ciertos tipos de operads,
entonces este espacio es homotópicamente equivalente a un espacio de lazos, iterado finita o infinitamente.

Intuitivamente, el uso de un operad puede verse como una generalización de la relación entre un grupo y sus
representaciones, pero con la diferencia de que la estructura algebraica a describir puede que tenga muchas
operaciones de diferentes aridades y que estas operaciones no necesariamente cumplan con condiciones de
conmutatividad o asociatividad. Así, un operad $\mathcal{P}$ está conformado por diferentes tipos de objetos,
cuya naturaleza depende de la categoría en la que se esté trabajando. Estos objetos son los que organizan
las operaciones abstractas que luego serán realizadas como las operaciones que queremos describir.
Por ejemplo, $P(m)$ representaría las operaciones de nuestra entidad que tienen $m$ entradas y una salida.
Las condiciones que se le piden a estas familias de operaciones es que cuenten con una
composición asociativa y con un elemento unidad con respecto a esta operación. En el caso de los operads simétricos, 
también se necesitan condiciones de equivarianza con respecto a las acciones de los grupos simétricos.

Para efectos de este artículo, la teoría de operads será estudiada en el marco de las categorías de
módulos diferenciales graduados y se utilizará la definición clásica de operads formulada por May.
Luego, en la sección 4 se hará una descripción detallada de la construcción de operads libres simétricos y no simétricos,
la cual se puede interpretar como una generalización de las estructuras libres usuales asociadas a grupos, anillos,
espacios vectoriales, etc. 

Para realizar la construcción de operads libres se recurre a un punto de vista categórico, donde los operads
son descritos como cierto tipo de monoides. Esto nos lleva a tener una construcción libre más cercana a la
intuición, pues con los productos tensoriales asociados, las interpretaciones de los objetos como operaciones 
de varias entradas son más claras.
En la última parte se estudian las principales adjunciones entre los funtores libres de operads libres simétricos
y no simétricos, con la idea de describir el caso simétrico a través del caso no simétrico. 

El presente trabajo está basado en el segundo capítulo de mi tesis doctoral (\cite{sanchez-jesus-thesis-2016}), 
donde esta manera de abordar la construcción de operads libres lleva al diseño de operads del tipo
$E_\infty$, los cuales son usados para estudiar las propiedades homotópicas de estructuras asociadas
a complejos de cadenas, como las descritas por Alain Prouté en \cite{alain-transf-EM-1983} y \cite{alain-transf-AW-1984}.

\section{Preliminares}
		
	Se utilizarán libremente el lenguaje y las propiedades de los módulos diferenciales graduados
	que aparecen en $\cite{loday2012algebraic}$ y $\cite{sanchez-jesus-thesis-2016}$.
	Con respecto a la teoría de categorías, las terminologías y propiedades son tomadas 
	de $\cite{MacLane-1998}$ y $\cite{alain-logique-cat}$.	

	Usaremos $F$ para denotar un cuerpo, el cual puede ser 
	$\mathbb{Z}/p\mathbb{Z}$, donde $p$ es un número primo, o incluso $\mathbb{Q}$.
	Los módulos sobre $F$ serán llamados simplemente módulos.
	Se denota $[n]$ al conjunto $\{1,\ldots,n\}$, donde $n$ es un entero positivo.
	El grupo simétrico formado por las permutaciones de $[n]$ se escribe $\Sigma_n$.

	Un módulo $M$ se dirá módulo graduado si existe una familia $\{M_i\}_{i\in \mathbb{Z}}$
	de submódulos de $M$, tales que $M=\bigoplus_{i\in \mathbb{Z}}M_i$.	
	Un módulo diferencial graduado o $\dg$-módulo 
	es un módulo graduado $M$ junto con un morfismo $\partial:M\to M$
	de grado $-1$, tal que $\partial \circ \partial=0$.	
	Una aumentación de un $\dg$-módulo es un morfismo de $\dg$-módulos 
	$\epsilon:M\to \field$ de grado $0$. Similarmente, una coaumentación de $M$
	es un morfismo de $\dg$-módulos $\eta:\field\to M$ de grado $0$.
	Si un $\dg$-módulo $M$ tiene una aumentación $\epsilon$ y una coaumentación $\eta$,
	tal que $\epsilon \circ \eta=1_\field$, entonces $M$ se dice módulo diferencial
	graduado con aumentación o $\dga$-módulo. Se denota la categoría de $\dga$-módulos 
	como $\dgacat$. 

\section{Operads simétricos}

	Los operads simétricos u operads, son colecciones de objetos que tienen asociados a sus componentes
	acciones de los grupos simétricos $\Sigma_n$.
	Debido a que se trabajará en un contexto graduado, 
	los signos de los diagramas incluidos están determinados por la convención de Koszul y
	son omitidos para facilitar la lectura. Recuerde que la convención de Koszul dice que, 
	si la posición de dos símbolos en una expresión, 
	de grados $p$ y $q$, es permutada,
	la expresión resultante será multiplicada por $(-1)^{pq}$.

\begin{definition}
\label{df-operad}
\index{operad@Operad}
	Un operad $\mathcal{P}$ es una colección de $\dga$-módulos $\{\mathcal{P}(n)\}_{n\geq 0}$
	junto con:
	\begin{enumerate}
	\item Un morfismo $\eta:\field\to \mathcal{P}(1)$, llamado la unidad de $\mathcal{P}$. 
	\item Para cada $n$, una acción a la derecha por el grupo simétrico $\Sigma_n$ sobre $\mathcal{P}(n)$,
	es decir, un morfismo de $\dga$-módulos que hace de $P(n)$ un DGA-$\field[\Sigma_n]$-módulo derecho.
		\begin{equation}
			\begin{gathered}
			P(n)\otimes \field[\Sigma_n] \longrightarrow P(n)
			\end{gathered}
		\end{equation}
	\item Por cada tupla $(h,i_1,\ldots,i_h)$, un morfismo de $\dga$-módulos,
		\begin{equation}
			\begin{gathered}
			\gamma_{(h,i_1,\ldots,i_h)}: \mathcal{P}(h)\otimes \mathcal{P}(i_1)\otimes 
			\cdots \otimes \mathcal{P}(i_h)\to \mathcal{P}(n)
			\end{gathered}
		\end{equation}
	donde $n=i_1+\cdots +i_h$ y $n,h,i_j\geq 0$. Todo morfismo de este tipo se escribirá $\gamma$.
	\end{enumerate}
	Estos morfismos deben de cumplir las siguientes condiciones.
	\begin{enumerate}
	\item Los morfismos $\gamma$ son asociativos, en el sentido del siguiente diagrama conmutativo.
		\begin{equation}
		\label{dg-associ-operad}
		\begin{gathered}
		\xymatrix{
		P(h)\otimes \left[ \bigotimes_{p=1}^{h}P(i_p) \right]\otimes \left[\bigotimes_{p=1}^{h}\bigotimes_{q=1}^{i_p}P(r_{p,q})
		\ar[r]^-{\gamma \otimes 1}\right] \ar[dd]_-{\text{intercambio}} 
		& P(n)\otimes \left[ \bigotimes_{p=1}^{h}\bigotimes_{q=1}^{i_s}P(r_{p,q})\right]
		\ar[d]^-{\gamma}\\
		& P(r)\\
		P(h)\otimes \bigotimes_{p=1}^{h} \left[P(i_p)\otimes \bigotimes_{q=1}^{i_p}P(r_{p,q})\right]
		\ar[r]_-{1\otimes \gamma^{\otimes h}} 
		& P(h)\otimes \bigotimes_{p=1}^{h}P(r_p)  \ar[u]_-{\gamma}
		}
		\end{gathered}
		\end{equation}
	Donde $n=\sum_{p=1}^{h}i_p$, $r=\sum_{p=1}^{h}\sum_{q=1}^{i_p}r_{p,q}=\sum_{p=1}^{h}r_p$
	y la flecha vertical izquierda es solamente un intercambio de factores.
	\item La unidad $\eta:\field  \to P(1)$ hace conmutativos los siguientes diagramas.
		\begin{equation}
		\label{dg-unit-operad}
		\begin{gathered}
		\xymatrix@R=25pt{
		\mathcal{P}(n)\otimes \field^{\otimes n} \ar[r]^-{\cong} \ar[d]_-{1\otimes \eta^{\otimes n}}& \mathcal{P}(n)\\
		\mathcal{P}(n)\otimes \mathcal{P}(1)^{\otimes n} \ar[ru]_-{\gamma} &
		}
		\xymatrix{
		\field\otimes \mathcal{P}(n) \ar[r]^-{\cong} \ar[d]_-{\eta \otimes 1}& \mathcal{P}(n)\\
		\mathcal{P}(1)\otimes \mathcal{P}(n)\ar[ru]_-{\gamma} &
		}
		\end{gathered}
		\end{equation}
	\item La acción de los grupos simétricos deben satisfacer condiciones de equivarianza,
	expresadas por los siguientes diagramas conmutativos.
		\begin{equation}
		\label{dg-equiv-1-op}
		\begin{gathered}
		\xymatrix@C=5pc{
		\mathcal{P}(h)\otimes\mathcal{P}(i_1)\otimes \cdots \otimes\mathcal{P}(i_h)
		\ar[d]_-{\sigma \otimes \sigma^{-1}} \ar[r]^-{\gamma}
		& P(n)\ar[d]^-{\sigma(i_{1},\ldots,i_{n})}\\
		\mathcal{P}(h)\otimes\mathcal{P}(i_{\sigma(1)})\otimes \cdots \otimes \mathcal{P}(i_{\sigma(h)})
		\ar[r]^-{\gamma} & P(n)\\
		}
		\end{gathered}
		\end{equation}
	Donde $n=i_1+\cdots +i_h$ y la flecha $\sigma\otimes \sigma^{-1}$ consiste en la acción derecha de
	$\sigma$ sobre $P(h)$ y la acción izquierda de $\sigma^{-1}$ sobre el producto tensorial
	$P(i_1)\otimes \cdots \otimes P(i_h)$.
		\begin{equation}\label{dg-equiv-2-op}
		\begin{gathered}
		\xymatrix@C=5pc{
		\mathcal{P}(h)\otimes \mathcal{P}(i_1)\otimes \cdots \otimes\mathcal{P}(i_h)
		\ar[r]^-{1\otimes \tau_1\otimes \cdots \otimes \tau_h} \ar[d]_-{\gamma} & 
		\mathcal{P}(h)\otimes\mathcal{P}(i_1)\otimes \cdots \otimes\mathcal{P}(i_h)
		\ar[d]^-{\gamma}\\
		\mathcal{P}(n)\ar[r]^-{\tau_1\oplus \cdots \oplus \tau_n} & \mathcal{P}(n)
		}
		\end{gathered}
		\end{equation}
	Donde $n=i_1+\cdots +i_h$ y la acción $1\otimes \tau_1\otimes \cdots \otimes \tau_h$ 
	es la identidad de $P(h)$ sobre el primer factor y la acción derecha de $\tau_j$ en el factor $P(i_j)$.
	\end{enumerate}
\end{definition}

	Un operad fundamental que se puede ver como el paradigma para diseñar el concepto de
	operad, es el operad de endomorfismos.
\begin{definition}
\label{df-endo-operad}
\index{endomorphismoperad@Endomorphism Operad}
	Para cada $M\in \dgacat$, el operad $End(M)$ de endomorfismos de $M$ se define como:
	\begin{enumerate}
	\item Para todo $n\geq 0$, $End(M)(n)=Hom(M^{\otimes n},M)$, es decir el $\dga$-módulo de
	aplicaciones homogéneas de  $M^{\otimes n}$ a $M$.
	\item La unidad $\eta:\field \to End(M)(1)$ es la identidad de $M$. 
	\item La acción derecha de $\Sigma_n$ sobre $End(M)$ es inducida por la acción izquierda de $\Sigma_n$ 
	sobre $M^{\otimes n}$, es decir,
	$f\sigma(x_1\otimes \cdots \otimes x_n)=
	\|\sigma\|f(x_{\sigma^{-1}(1)}\otimes \cdots \otimes x_{\sigma^{-1}(n)})$. Donde $\|\sigma\|$
	es el signo de la permutación $\sigma$.
	\item La composición de aplicaciones,
		\begin{equation}
		\begin{gathered}
		\gamma:End(M)(h)\otimes End(M)(i_1)\otimes \cdots \otimes End(M)(i_h)\to End(M)(n)
		\end{gathered}
		\end{equation}
		donde $n=i_1+\cdots +i_h$, está dada por,
		\begin{equation}
		\begin{gathered}
		\gamma (f_h\otimes f_{i_1}\otimes \cdots \otimes f_{i_h})=f_h \circ(f_{i_1}\otimes \cdots \otimes f_{i_h})
		\end{gathered}
		\end{equation}
	\end{enumerate}
\end{definition}

El lector puede verificar que $End(M)$ satisface
las condiciones de la definición \ref{df-operad}.
\begin{definition}
\label{df-operads-morphism}
\index{operadmorphism@Operad morphism}
\label{df-category-operads}
\index{categoryoperads@Category of operads}
	Sean $P$ y $Q$ dos operads. 
	Un morfismo $f$ de $P$ a $Q$, es una colección de morfismos de $\dga$-módulos,
		\begin{equation}
		\begin{gathered}
		f_n:P(n)\to Q(n)
		\end{gathered}
		\end{equation}
	que satisfacen las siguientes condiciones.
	\begin{enumerate}
	\item El morfismo $f_1:P(1)\to Q(1)$ preserva la unidad de los operads, es decir, $f_1\eta=\eta$.
		\begin{equation}
		\begin{gathered}
		\xymatrix{
		P(1)\ar[rr]^-{f_1} && Q(1)\\
		&\field \ar[ul]^-{\eta} \ar[ur]_-{\eta}& 
		}
		\end{gathered}
		\end{equation}
	\item Los morfismos $f_n:P(n)\to Q(n)$ son $\Sigma_n$-equivariantes, es decir, el siguiente
	diagrama es conmutativo para cada $\sigma \in \Sigma_n$.
		\begin{equation}
		\begin{gathered}
		\xymatrix{
		P(n) \ar[r]^-{f_n} \ar[d]_-{\sigma}& Q(n)\ar[d]^-{\sigma}\\
		P(n) \ar[r]^-{f_n} & Q(n)\\
		}
		\end{gathered}
		\end{equation}
	\item $f$ preserva las operaciones de composición de los operads, 
	esto es, que el siguiente diagrama es conmutativo.
		\begin{equation}
		\begin{gathered}
		\xymatrix@C=3cm@R=1cm{
		P(h)\otimes P(i_1)\otimes \cdots \otimes P(i_h) \ar[r]^-{\gamma_P} 
		\ar[d]_-{f_h\otimes f_{i_1}\otimes \cdots \otimes f_{i_h}}
		& P(n) \ar[d]^-{f_n}\\
		Q(h)\otimes Q(i_1)\otimes \cdots \otimes Q(i_h) \ar[r]^-{\gamma_Q} 
		& Q(n) \\
		}
		\end{gathered}
		\end{equation}
	\end{enumerate}
La categoría de operads sobre $\dgacat$ se denota $\mathcal{OP}$.
\end{definition}


\begin{example}
\label{N-operad}
\index{operadn@Operad $\mathcal{N}$}
	El operad $\mathcal{N}$ está dado por $\mathcal{N}(n)=\field$ para cada $n$ no negativo,
	donde $\field$ es visto como un $\dga$-módulo concentrado en grado cero.  
	La unidad $\eta$ es la identidad de $\field$, $\Sigma_n$ actúa de manera trivial en
	cada componente y las composiciones $\gamma$, si denotamos $a_i$ al generador de grado cero de $N(i)$,
	están dadas por la regla,
		\begin{equation}
		\begin{gathered}
		\gamma:a_h\otimes a_{i_1}\otimes \cdots \otimes a_{i_h}\to a_{n}
		\end{gathered}
		\end{equation}
		donde $n=i_1+\cdots +i_h$. 
\end{example}

\begin{example}
\label{M-operad}
\index{operadm@Operad $\mathcal{M}$}
	Al hacer libre la acción de los grupos simétricos en el ejemplo anterior, se produce un operad 
	que denotamos $\mathcal{M}$. Las componentes de $\mathcal{M}$ son los
	módulos concentrados en grado cero $M(n)=\field[\Sigma_n]$,
	para cada $n$ no negativo. 
	Las composiciones $\gamma$ son determinadas, al igual que antes,
	por los generadores de grado cero, pero respetando las acciones de los grupos simétricos:
		\begin{equation}
		\begin{gathered}
		\gamma(a_h\otimes a_{i_1}\sigma_{i_1} \otimes \cdots \otimes  a_{i_h}\sigma_{i_1})= 
		a_{n}(\sigma_{i_1}\oplus \cdots \oplus \sigma_{i_h}) 
		\end{gathered}
		\end{equation}
	y
		\begin{equation}
		\begin{gathered}
		\gamma(a_h \sigma \otimes a_{i_{\sigma^{-1}(1)}}\otimes \cdots \otimes a_{i_{\sigma^{-1}(h)}})= 
		a_{n} \sigma(i_1,\ldots, i_h)
		\end{gathered}
		\end{equation}
	donde $n=i_1+\cdots +i_h$. 
\end{example}

\section{Adjunciones}\label{sec-adjunctions}

\begin{definition}\label{df-adjoint-functors}
\index{adjointfunctors@Adjoint functors $F\vdash G$}
\index{bijection-adjoint@Bijection associated to $F\vdash G$}
Sean $F:\mathcal{C}\to\mathcal{D}$ y $G:\mathcal{D}\to \mathcal{C}$ dos funtores.
$F$ se dice adjunto a la izquierda de $G$, se denota por $F\vdash G$, 
si existe una biyección $\theta$,
	\begin{equation}\label{form-bij-adj-def}
		\begin{gathered}
		\xymatrix{
		\mathcal{D}(F(X),Y) \ar[r]^-{\theta} & 	\mathcal{C}(X,G(Y))
		}
		\end{gathered}
	\end{equation}
natural en $X$ y $Y$.
\end{definition}

\begin{definition}\label{df-unit-counit}
\index{unitadjunction@Unit of an adjunction}
\index{counitadjunction@Counit of an adjunction}
Sean $F\vdash G:\mathcal{C}\to \mathcal{D}$ un par de funtores adjuntos.
	\begin{enumerate}
	\item La transformación natural $\eta:1_\mathcal{C}\to GF$ dada por
	$\eta_X=\theta(1_{F(X)})$, es llamada unidad de la adjunción.
	\item La transformación natural $\epsilon:FG\to 1_{\mathcal{D}}$
	dada por $\epsilon_Y=\theta^{-1}(1_{G(Y)})$, es llamada la counidad de la adjunción.
	\end{enumerate}
\end{definition}

\begin{proposition}\label{prop-triangular-eq-unit-counit}
\index{triangularequations@Triangular equations}
Sean $F\vdash G:\mathcal{C}\to \mathcal{D}$ un par de funtores adjuntos,
con unidad $\eta$ y counidad $\epsilon$.
Entonces los siguientes diagramas son conmutativos.
	\begin{equation}\label{diagram-triang-equ-unit-counit}
	\begin{gathered}
		\xymatrix@R=3pc@C=3pc{
			F \ar[rd]_-{1_F} \ar[r]^-{F\eta} & FGF \ar[d]^-{\epsilon_F}\\
			& F
			}
			\,\,\,\,\,\,\,\,\,\,\,\,\,\,\,\,\,\,\,\,\,
			\xymatrix@R=3pc@C=3pc{
			GFG \ar[d]_-{G\epsilon}& G \ar[l]_-{\eta_G} \ar[dl]^-{1_G}\\
			G
		}
	\end{gathered}
	\end{equation}
Las ecuaciones que se desprenden de la conmutatividad de estos diagramas,
	\begin{align}\label{fm-triangular-equ}
		&\epsilon_F \cdot F\eta=1_F\\
		&G\epsilon \cdot \eta_G=1_G
	\end{align}
son llamadas ecuaciones triangulares.
\end{proposition}

\begin{proof}
Sea $X$ un objeto de $\mathcal{C}$. Para la conmutatividad del primer diagrama
tenemos que mostrar que $1_{F(X)}=\epsilon_{F(X)}F(\eta_X)$.
Por definición $\eta_X=\theta(1_{F(X)})$, entonces al usar el hecho que $\theta$ es una biyección,
solo tenemos que chequear que $\theta(\epsilon_{F(X)}F(\eta_X))=\eta_X$.
	\begin{align*}
		\theta(\epsilon_{F(X)}F(\eta_X)) &=
		\theta(\epsilon_{F(X)})\theta(F(\eta_X)) \\
		&= \theta(\epsilon_{F(X)})\eta_X 
		&& \text{(by naturality of $\theta$)} \\
		&= \theta(\theta^{-1}(1_{GF(X)}))\eta_X \\
		&= \eta_X
	\end{align*}
La conmutatividad del otro diagrama es similar.
\end{proof}

\begin{definition}\label{df-universal-arrow}
\index{universalarrow@Universal arrow}
Sea $G:\mathcal{D}\to \mathcal{C}$ un funtor covariante y  $X$ un objeto de $\mathcal{C}$. 
Una flecha universal desde $X$ a $G$ es un morfismo de la forma $\psi:X\to GF(X)$ tal que,
para cada morfismo $f:X \to G(Y)$ hay un único morfismo de  $\mathcal{D}$, denotado $\theta^{-1}(f)$, 
desde $F(X)$ a $Y$, que satisface $G(\theta^{-1}(f))\eta=f$. 
	\begin{equation}\label{diagram-univ-arrow}
	\begin{gathered}
		\xymatrix{
		X \ar[dr]_-{f}\ar[r]^-{\eta} & GF(X)\ar[d]^-{G(\theta^{-1}(f))}\\
		& G(Y)
		}
	\end{gathered}
	\end{equation}
\end{definition}

\begin{theorem}\label{th-adjunction-exists-by-univ-arrow}
Sea $G:\mathcal{D}\to \mathcal{C}$ un functor tal que para cada objeto $X\in \mathcal{C}$ existe una flecha universal
$\eta_X:X\to G(F(X))$. 
Entonces la aplicación $F$ entre los objetos de $\mathcal{C}$ 
y $\mathcal{D}$, se extiende se manera única en un funtor $F:\mathcal{C}\to \mathcal{D}$ tal que $F\vdash G$.
\end{theorem}\qed

\begin{remark}
Para una prueba del teorema \ref{th-adjunction-exists-by-univ-arrow} ver \cite{MacLane-1998}, $\S 4$ teorema 2. 
Sin embargo, la extensión de $F$ para las flechas, se hace de la siguiente manera:
sea $f:X\to Y$ un morfismo en $\mathcal{C}$ y considere el siguiente diagrama.

	\begin{equation}\label{diagram-univ-arrow-nat}
	\begin{gathered}
		\xymatrix{
		X \ar[d]_-f \ar[r]^-{\eta_X} & G(F(X)) \ar@{-->}[d]^-{G(F(f))}\\
		Y  \ar[r]^-{\eta_Y} & G(F(Y))
		}
	\end{gathered}
	\end{equation}

La existencia y unicidad de $F(f)$, el cual hace el diagrama conmutativo,
es garantizada por la propiedad universal de $\eta_X$. 
\end{remark}

\begin{theorem}\label{th-pres-colim-lim-adj}
Todo funtor adjunto a la izquierda preserva colímites y todo funtor adjunto a la derecha preserva límites.
and every right adjoint preserves limits.
\end{theorem}\qed

\begin{remark}
Para una prueba ver $\S$5 de \cite{MacLane-1998}, o $\S 9$ de \cite{awodey2006category}.
\end{remark}


\section{Coecualizadores reflexivos}\label{sec-reflex-coeq}

\index{coequalizer@Coequalizer $Ceq(f,g)$}
En una categoría $\mathcal{C}$, el coecualizador de dos morfismos $f,g:X\to Y$
es el colímite sobre el diagrama formado por ellos. Se denota $Ceq(f,g)$. 

	\begin{equation}\label{dgm-coequalizer}
	\begin{gathered}
		\xymatrix{
		X \ar@<+0.5pc>[r]^-f \ar@<-0.5pc>[r]_-g& Y \ar[r]^-q& Ceq(f,g)
		}
	\end{gathered}
	\end{equation}

Equivalentemente, el coecualizador de $f$ y $g$ es un objeto inicial en la categoría de morfismos $l$ 
que ecualizan a la izquierda $f$ y $g$, es decir, $l\circ f=l\circ g$.
Nos interesa un tipo especial de coecualizadores, llamados coecualizadores reflexivos.
Ellos juegan un rol importante en la prueba de la existencia de pequeños colímites en la categoría de operads.

\begin{definition}\label{df-reflexive-diagram}
\index{reflexivepair@Reflexive pair}
Sea $\mathcal{D}_0$ la categoría generada por el diagrama,
	\begin{equation}
	\begin{gathered}
		\xymatrix{
		x_0 \ar@<+0.3pc>[rr]^-{i} \ar@<-0.3pc>[rr]_-{j}& & x_1 \ar@<-0.2pc>@/_1pc/[ll]_-{s}
		}
	\end{gathered}
	\end{equation}
donde las flechas satisfacen $is=1=js$. Para cada categoría $\mathcal{C}$, 
llamamos par reflexivo de $\mathcal{C}$, a cualquier diagrama en $\mathcal{C}$ de la categoría $\mathcal{D}_0$.
Es decir, un par reflexivo es un par de flechas paralelas que tienen una sección en común.
\end{definition}

\begin{proposition}\label{prop-reflexive-coeq}
Sean $f,g:X\to Y$ dos morfismos en la categoría $\mathcal{C}$.
Si existe un morfismo $s:Y\to X$ en $\mathcal{C}$ tal que $f\circ s=g\circ s=1_Y$ 
entonces el coecualizador de $f$ y $g$ (si existe) es isomorfo al colímite sobre el diagrama formado por $f$, $g$ y $s$.
\end{proposition}

\begin{proof}
Es claro que $q$ en el diagrama \ref{dgm-coequalizer} es un epimorfismo.
Sea $(B,\alpha:X\to B,\beta:Y\to B)$ el colímite sobre el diagrama formado por $f$, $g$ y $s$.
Entonces $\alpha$ y $\beta$ satisfacen $\beta f=\alpha=\beta g$ y $\alpha s=\beta$. También tenemos que 
$\alpha$ es un epimorfismo. De hecho, si $r,s:B\to Z$ son dos flechas tal que $r\alpha=s \alpha$
entonces $(Z,r\alpha f:X\to Z,r\alpha:Y\to Z)$ es un cocono sobre $f$, $g$ y $s$ (ya que $fs=1$),
lo cual implica por la propiedad universal de colímites que $r\alpha$ se factoriza de manera única por $r$
a traves de $\alpha$. Lo mismo aplica para $s\alpha$, pero $r\alpha=s\alpha$ entonces $r=s$ y
$\alpha$ es un epimorfismo. Para mostrar que $B$ y $Ceq(f,g)$ son isomorfos, primero note que
$\alpha$ ecualiza a la izquierda a $f$ y $g$. En efecto, $\alpha f=\beta=\alpha g$. Entonces existe una
única flecha $h:Ceq(f,g)\to B$ tal que $hq=\alpha$. Ahora, $(Ceq(f,g),qf:X\to Ceq(f,g),q:Y\to Ceq(f,g))$
es un cocono sobre $f$, $g$ y $s$, ya que $qfs=q1=q$. Entonces existe una única flecha
$\overline{h}:B\to Ceq(f,g)$ tal que $\overline{h}\alpha=q$ y $\overline{h}\beta=fq$.
Pero $q$ y $\alpha$ son epimorfismos, así que tenemos que $h\overline{h}\alpha=hq=\alpha$ implica
$h\overline{h}=1$  y que $\overline{h}hq=\overline{h}\alpha=q$ implica $\overline{h}h=1$.
Por lo tanto, $B$ y $Ceq(f,g)$ son isomorfos.
\end{proof}

\begin{remark}
La proposición \ref{prop-reflexive-coeq} dice que el morfismo $s$ no cambia el coecualizador.
\end{remark}

\begin{definition}\label{df-reflexive-coeq}
\index{reflexicecoequalizer@Reflexive coequalizer}
Considere $f,g:X\to Y$ dos morfismos en una categoría $\mathcal{C}$.
Si existe un morfismo $s:Y\to X$ in $\mathcal{C}$ tal que $f\circ s=g\circ s=1_Y$ entonces
$Ceq(f,g)$ es llamado coecualizador reflexivo de $f$ y $g$.
\end{definition}

\begin{definition}\label{final-functor}
\index{finalfunctor@Final functor}
Considere $F:C\to D$ un funtor covariante.
$F$ se dice final si satisface las siguientes condiciones para cada objeto $X\in \mathcal{D}$.
\begin{enumerate}
\item Existe un morfismo desde $X$ hasta un objeto de la forma $F(Y)$.
\item Para cada par de tales morfismos desde $X$, $\alpha:X\to F(Y)$ y $\alpha':X\to F(Y')$,
existe una secuencia finita  $g_1,\ldots,g_k$ de morfismos de $\mathcal{C}$
haciendo el siguiente diagrama conmutativo.

\begin{equation}
\begin{gathered}
\xymatrix{	
&&X \ar[dll]_-{\alpha} \ar[dl] \ar[dr] \ar[drr]^-{\alpha'}&&\\
F(Y)\ar[r]_-{F(g_1)} & F(Y_1) \ar@{<-}[r]_-{F(g_2)}& \cdots &F(Y_{k-1}) \ar@{<-}[l]^-{F(g_{k-1})}
\ar@{<-}[r]_-{F(g_k)}& F(Y')
}
\end{gathered}
\end{equation}
\end{enumerate}
\end{definition}

\begin{remark}
Otra forma de definir un funtor $F:\mathcal{C}\to \mathcal{D}$ como final es diciendo que para cada
$X\in \mathcal{D}$ la coma categoría $X/F$ es no vacía y conexa (ver \cite{MacLane-1998}, $\S$9).
\end{remark}

\begin{proposition}\label{prop-D0-final}
Considere la categoría $\mathcal{D}_0$ de la definición \ref{df-reflexive-diagram}.
Entonces, para cada $n\geq 1$, el funtor diagonal desde $\mathcal{D}_0$
a la categoría producto $\mathcal{D}_0^n$, es final.
\end{proposition}

\begin{proof}
Sea $D:\mathcal{D}_0\to \mathcal{D}_0^n$ el funtor diagonal.
Tome $X$ un objeto de $\mathcal{D}_0^n$, entonces tiene la forma $X=(x_{i_1},\ldots,x_{i_n})$,
con $i_j\in{0,1}$. Existe un morfismo $f$ de $X$ a $D(x_1)$
dado por $f=(f_{i_1},\ldots,f_{i_n})$, donde $f_{i_j}=1_{x_1}$ si $i_j=1$ y $f_{i_j}=f$ si $i_j=0$.
Note que aún se puede tener un morfismo de $X$ a $D(x_1)$, al tomar arbitrariamente $f$ de $g$ en las entradas
$f_{i_j}$ cuando $i_j=0$.
Pero el único morfismo de $X$ a $D(x_0)$ está dado por $s=(s_{i_1},\ldots,s_{i_n})$,
donde $s_{i_j}=1_{x_0}$ si $i_j=0$ y $s_{i_j}=s$ si $i_j=1$.
Ahora se verifica la segunda condición en la definición \ref{final-functor}.
Tome $\alpha,\alpha'$ dos morfismos de  $X$ a $D(x_1)$. Sea
$\beta=(b_{i_1},b_{i_k})$ el morfismo de $X$ a $D(x_0)$ 
definido por $b_{i_j}=s$ si $i_j=1$ y $b_{i_j}=1_{x_0}$ si $i_j=0$.
Entonces el siguiente diagrama es conmutativo.

\begin{equation}
\begin{gathered}
\xymatrix{
&X \ar[ld]_-{\alpha} \ar[rd]^-{\alpha'} \ar[d]^-{\beta}&\\
D(x_1) \ar[r]_-{D(s)}& D(x_0)& D(x_1)\ar[l]^-{D(s)}
}
\end{gathered}
\end{equation}
Esto es suficiente para mostrar que $D$ es un funtor final.
\end{proof}

Los funtores finales son útiles para calcular colímites, como lo muestra la siguiente proposición.
Para una prueba, ver \cite{MacLane-1998}, $\S$9.

\begin{proposition}\label{prop-final-functors}
Considere $F:\mathcal{D}\to \mathcal{C}$ un diagrama en la categoría $\mathcal{C}$ y $I:\mathcal{D'}\to \mathcal{D}$
un funtor final tal que el colímite de $F\circ I$ existe. 
Entonces el colímite de $F$ existe y es canónicamente isomorfo al colímite de $F\circ I$.
\end{proposition}\qed

\section{Colimites de operads}\label{sec-amalgamated Sum Operads}

En esta sección vamos a mostrar que la categoría de operads tiene todos colímites 
sobre diagramas pequeños. 

El producto tensorial de $\dga$-módulos preserva todos los colímites pequeños en cada componente,
entonces satisface el siguiente lema (\cite{fresse-homoperad-2016}).

\begin{lemma}\label{lemma-preserve-coeq}
Sea $F:\mathcal{C}^{n}\to \mathcal{C}$ un funtor covariante.
Si $F$ preserva coecualizadores reflexivos en cada componente,
entonces $F$ preserva coecualizadores reflexivos en $C^n$. 
Es decir, si para cada $1\leq i\leq n$ y cada diagrama reflexivo $X_i \coeq Y_i$ en $\mathcal{C}$,
el morfismo dado por la propiedad universal de coecualizadores 
desde el coecualizador del diagrama en $\mathcal{C}$,

\begin{equation}
\begin{gathered}
F(A_1,\ldots,A_{i-1},X_i,A_{i+1},\ldots,A_n)\coeq F(A_1,\ldots,A_{i-1},Y_i,A_{i+1},\ldots,A_n)
\end{gathered}
\end{equation}

a $F(A_1,\ldots,A_{i-1},Ceq(X_i\coeq Y_i),A_{i+1},\ldots,A_n)$,
es un isomorfismos,
entonces para cada colección de diagramas reflexivos $\{X_i\coeq Y_i\}_{1\leq i\leq n}$
el morfismo desde el coecualizador del diagrama en $\mathcal{C}$,
\begin{equation}
\begin{gathered}
F(X_1,\ldots,X_n)\coeq F(Y_1,\ldots,Y_n)
\end{gathered}
\end{equation}
a $F(Ceq(X_1\coeq Y_1),\ldots,Ceq(X_n\coeq Y_n))$
es un isomorfismo.
\end{lemma}

\begin{proof}
La colección de diagramas reflexivos $\{X_i\coeq Y_i\}_{1\leq i\leq n}$
define una colección de funtores $\{T_i:\mathcal{D}_0\to \mathcal{C}\}_{1\leq i\leq n}$.
Se usa la notación $\underset{\alpha \in \mathcal{D}_0}{\text{colim}\,}T_i(\alpha)$ para $Ceq(X_i \coeq Y_i)$.
Así, la hipotesis puede escribirse como
\begin{equation}
\begin{gathered}
\underset{\alpha \in \mathcal{D}_0}{\text{colim}\,}F(A_1,\ldots,A_{i-1},T_i(\alpha),A_{i+1},\ldots,A_n)
\underset{\longrightarrow}{\cong} 
F(A_1,\ldots,A_{i-1},\underset{\alpha \in \mathcal{D}_0}{\text{colim}\,}T_i(\alpha),A_{i+1},\ldots,A_n)
\end{gathered}
\end{equation}

Por proposición \ref{prop-D0-final}, la diagonal $D:\mathcal{D}_0\to \mathcal{D}_0^n$
es un funtor final. Considere el funtor $T_1\times \cdots \times T_n:\mathcal{D}_0^n\to \mathcal{C}^n$.
Entonces, proposición \ref{prop-final-functors} dice que existe un isomorfismo
desde el colímite de $F(T_1\times \cdots \times T_n)D$ al colímite de $F(T_1\times \cdots \times T_n)$,
y tenemos

\begin{align}
&Ceq(F(X_1,\ldots,X_n)\coeq F(Y_1,\ldots,Y_n)) \nonumber\\
&=\underset{\alpha\in \mathcal{D}_0}{\text{colim}\,}
F(T_1\times \cdots \times T_n)D(\alpha)\\
&\underset{\longrightarrow}{\cong}\underset{(\alpha_1,\ldots,\alpha_n)\in \mathcal{D}_0^n}{\text{colim}\,}F(T_1(\alpha_1),
\ldots,T_n(\alpha_n))&& (\text{by \ref{prop-final-functors}})\\
&\cong \underset{\alpha_1\in \mathcal{D}_0}{\text{colim}\,}\ldots \underset{\alpha_n\in \mathcal{D}_0}{\text{colim}\,}
F(T_1(\alpha_1),\ldots,T_n(\alpha_n))\\
& \underset{\longrightarrow}{\cong} F(\underset{\alpha_1\in \mathcal{D}_0}{\text{colim}\,}T_1(\alpha_1),\ldots,
\underset{\alpha_n\in \mathcal{D}_0}{\text{colim}\,}T_n(\alpha_n)) && (\text{by hypothesis})\\
&= F(Ceq(X_1\coeq Y_1),\ldots,Ceq(X_n\coeq Y_n)) \nonumber
\end{align}
\end{proof}

En la siguiente proposición se construye un operad usando la definición clásica de operads en \ref{df-operad}.

\begin{proposition}\label{prop-coeq-reflex-OP}
En la categoría $\mathcal{OP}$,  el funtor de olvido $U:\mathcal{OP}\to \mathbb{S}$-Mod 
crea coecualizadores reflexivos.
\end{proposition}

\begin{proof}
Sea $\mathcal{P}\coeq \mathcal{Q}$ un par reflexivo en $\mathcal{OP}$. 
Se contruirá el coecualizador reflexivo $\mathcal{O}$ del diagrama en $\mathcal{OP}$. 
Para eso, primero se definen los componentes del operad $\mathcal{O}$ como $\mathcal{O}(n)=Coeq(P(n)\coeq Q(n))$.
Este coecualizador existe debido a que $\mathbb{S}$-Mod como $\dga$-Mod
tienen todos los pequeños límites. 
Para definir la composición $\gamma$ de $\mathcal{O}$, considere el siguiente morfismo de $\dga$-módulos.

\begin{equation}
\begin{gathered}
\begin{tikzpicture}
\foreach \h in {2}{
	\foreach \l in {0.5}{
	\node at (0,\h) {$Coeq\left[P(k)\otimes P(i_1)\otimes \cdots P(i_k)\coeq Q(k)
	\otimes Q(i_1)\otimes \cdots \otimes Q(i_k)\right]$} ;
	\node at (0,0) {$Coeq\left[P(k)\coeq Q(k)\right] \otimes 
	Coeq\left[P(i_k)\coeq Q(i_k)\right]\otimes \cdots \otimes 
	Coeq\left[P(i_k)\coeq Q(i_k)\right]$};
	\draw [->](0,\h-\l) to (0,+\l);
	\node at (\l,\h/2) {$\psi$};
 }
}
\end{tikzpicture}
\end{gathered}
\end{equation}

Por lema \ref{lemma-preserve-coeq}, este morfismo es un isomorfismo, 
entonces, se puede tomar su inversa $\psi^{-1}$ y definir $\gamma$ como la siguiente composición.

\begin{equation}
\begin{gathered}
\begin{tikzpicture}
\foreach \h in {2}{
	\foreach \l in {0.5}{
	\node at (0,\h+1) {$O(k)\otimes O(i_1)\otimes \cdots \otimes O(i_k)=$};
	\node at (0,\h) {$Coeq\left[P(k)\coeq Q(k)\right] \otimes 
	Coeq\left[P(i_k)\coeq Q(i_k)\right]\otimes \cdots \otimes 
	Coeq\left[P(i_k)\coeq Q(i_k)\right]$};
	\node at (0,0) {$Coeq\left[P(k)\otimes P(i_1)\otimes \cdots P(i_k)\coeq Q(k)
	\otimes Q(i_1)\otimes \cdots \otimes Q(i_k)\right]$} ;
	\node at (0,-\h) {$Coeq \left[ 
	P(i_1+\cdots i_k) \coeq Q(i_1+\cdots i_k)
 	\right]=O(i_1+\cdots +i_k)$};
	\draw [->](0,\h-\l) to (0,+\l);
	\node at (\l,\h/2) {$\psi^{-1}$};
	\draw [->](0,-\l) to (0,-\h+\l);
	\node at (\l,-\h/2) {$\overline{\gamma}$};
 }
}
\end{tikzpicture}
\end{gathered}
\end{equation}

Donde $\overline{\gamma}$ 
es el morfismo inducido por la composición de los operads
$\mathcal{P}$ y $\mathcal{Q}$. Para definir la unidad de $\mathcal{O}$, 
considere el siguiente diagrama conmutativo obtenido de $\mathcal{P}\coeq \mathcal{Q}$ 
y las propiedades del coecualizador.

\begin{equation}
\begin{gathered}
\xymatrix@R=0.1pc{
& P(1)\ar@<1ex>[dd] \ar@<-1ex>[dd] \ar[rrd]^-{i_{P(1)}}\\
\field \ar[ur]^-{\eta_\mathcal{P}} \ar[dr]_-{\eta_\mathcal{Q}}& && O(1)\\
& Q(1) \ar@/_2pc/[uu] \ar[rru]_-{i_{Q(1)}}
}
\end{gathered}
\end{equation}

Entonces la unidad para $\mathcal{O}$ es definida como la composición $i_{P(1)}\eta_{\mathcal{P}}:\field \to O(1)$.
Note que la eleccióne $i_{Q(1)}\eta_{\mathcal{Q}}$ da el mismo resultado, 
como consecuencia de la existencia de las flechas reflexivas.
No es complicado verificar que $\mathcal{O}$ con esta estructura satisface
los axiomas de operads y la propiedad universal para coecualizadores.
\end{proof}

\begin{proposition}\label{prop-ext-morphism-coproduct}
Considere $\{P_i\}_{i\in I}$ y $\{Q_i\}_{i\in I}$ dos pequeñas colecciones de objetos en la
categoría $\mathcal{C}$ tal que los colímites
$\underset{i\in I}{\text{colim}\,} P_i$ y $\underset{i\in I}{\text{colim}\,} Q_i$ existan.
Denote por $i\in I$ el cocono de aristas $p_i:P_i \to \underset{i\in I}{\text{colim}\,} P_i$ y
$q_i:Q_i\to \underset{i\in I}{\text{colim}\,} Q_i$.
\begin{enumerate}
\item Toda colección de morfismos $f_i:P_i\to \underset{i\in I}{\text{colim}\,} Q_i$, $i\in I$, 
determina un morfismo
$\overline{f}:\underset{i\in I}{\text{colim}\,} P_i\to \underset{i\in I}{\text{colim}\,} Q_i$,
tal que $f_i=\overline{f}\circ p_i$ para cada $i\in I$.
\item Cada colección de morfismos $f_i:P_i\to Q_i$, $i\in I$, determina un morfismo
$f:\underset{i\in I}{\text{colim}\,} P_i\to \underset{i\in I}{\text{colim}\,} Q_i$,
tal que $q_i\circ f=f_i \circ p_i$ para cada $i\in I$.
\end{enumerate}
\end{proposition}

\begin{proof}
La colección $f_i:P_i\to \underset{i\in I}{\text{colim}\,} Q_i$, $i\in I$ 
exibe a $\underset{i\in I}{\text{colim}\,} Q_i$ 
como un cocono sobre el diagrama $\{P_i\}_{i\in I}$,
entonces por la propiedad universal de coproductos, 
existe un único morfismo en $\mathcal{C}$,
$\overline{f}:\underset{i\in I}{\text{colim}\,} P_i\to \underset{i\in I}{\text{colim}\,} Q_i$ 
tal que $f_i=\overline{f}\circ p_i$ para cada $i\in I$.

Para el segundo enunciado, componga cada $f_i:P_i\to Q_i$ con la respectiva arista del cocono
$q_i:Q_i\to \underset{i\in I}{\text{colim}\,} Q_i$, así se tiene la colección
$g_i:P_i\to \underset{i\in I}{\text{colim}\,} Q_i$, 
con $g_i=q_i\circ f_i$ para  $i\in I$.
Y aplicando la primera parte se obtiene que esto determina
$f=\overline{g}:\underset{i\in I}{\text{colim}\,} P_i\to \underset{i\in I}{\text{colim}\,} Q_i$.
\end{proof}

\begin{proposition}\label{prop-op-coproducts}
La categoría de operads tiene todos los pequeños límites.
\end{proposition}

\begin{proof}
Consideree $\{P_i\}_{i\in I}$ una colección pequeña de operadas. 
Con ella es posible construir un par reflexivo en $\mathcal{OP}$, 
y por proposición \ref{prop-coeq-reflex-OP}
su coecualizador reflexico existe en $\mathcal{OP}$. 
La úlitma parte de la prueba consiste en chequear que este coecualizador reflexico 
es el colímite de $\{P_i\}_{i\in I}$.

De $\{P_i\}_{i\in I}$ se obtiene en $\mathbb{S}$-Mod la colección
$\{U(P_i)\}_{i\in I}$, denote su colímite $\underset{i\in I}{\text{colim}\,}  U(P_i)$,
y $\alpha_i:U(P_i)\to \underset{i\in I}{\text{colim}\,}  U(P_i)$ su cocono de aristas.

Los morfismos $\alpha_i:U(P_i)\to \underset{i\in I}{\text{colim}\,}  U(P_i)$ 
inducen el morfismo
$UF(\alpha_i):UFU(P_i)\to UF(\underset{i\in I}{\text{colim}\,}  U(P_i))$, 
el cual determina el siguiente morfismo.

\begin{equation}
\begin{gathered}
\xymatrix{
\underset{i\in I}{\text{colim}\,}  UFU(P_i) \ar[r]^-{\underline{d_0}}&  UF(\underset{i\in I}{\text{colim}\,}  U(P_i))
}
\end{gathered}
\end{equation}

Considere la unidad $\epsilon$ y counidad $\eta$ de la adjunción $F\vdash U$,
y la siguiente composición en $\mathbb{S}$-Mod.

\begin{equation}\label{dg-d1-def}
\begin{gathered}
\xymatrix{
UFU(P_i) \ar[r]^{U(\epsilon_{P_i})} & U(P_i) \ar[r]^-{\alpha_i} & \underset{i\in I}{\text{colim}\,}  U(P_i) \ar[r]^-{\eta} & UF(\underset{i\in I}{\text{colim}\,}  U(P_i))
}
\end{gathered}
\end{equation}

Estas composiciones determinan el morfismo en $\mathbb{S}$-Mod,

\begin{equation}
\begin{gathered}
\xymatrix{
\underset{i\in I}{\text{colim}\,}  UFU(P_i) \ar[r]^{\underline{d_1}} & UF(\underset{i\in I}{\text{colim}\,}  U(P_i))
}
\end{gathered}
\end{equation}

Por la propiedad universal de operads libres $\underline{d_0}$ y $\underline{d_1}$
van a determinar los morfismos $d_0$ y $d_1$ en $\mathcal{OP}$ en el siguente diagrama conmutativo.

\begin{equation}
\begin{gathered}
\xymatrix@R=4pc@C=3pc{
UF \left( \underset{i\in I}{\text{colim}\,}  UFU(P_i) \right) \ar@<0.5ex>[rd]^{U(d_1)} \ar@<-0.5ex>[rd]_{U(d_0)} 
& F \left( \underset{i\in I}{\text{colim}\,}  UFU(P_i) \right) \ar@<0.5ex>[rd]^{d_1} \ar@<-0.5ex>[rd]_{d_0} &\\
\underset{i\in I}{\text{colim}\,}  UFU(P_i)\ar[u]^-{\eta} \ar@<0.5ex>[r]^{\underline{d_1}} \ar@<-0.5ex>[r]_{\underline{d_0}} 
& UF(\underset{i\in I}{\text{colim}\,}  U(P_i)) & F(\underset{i\in I}{\text{colim}\,}  U(P_i))
}
\end{gathered}
\end{equation}

Ahora se da la contracción  $s$ para $d_0$ y $d_1$. Con la counidad $\eta$ considere
los morfismos $\eta_{U(P_i)}:U(P_i)\to UFU(P_i)$. 
Por las propiedades del colímite
\ref{prop-ext-morphism-coproduct} ellos determinan un morfismo de $\mathbb{S}$-modules
$\beta:\underset{i\in I}{\text{colim}\,}  U(P_i)\to \underset{i\in I}{\text{colim}\,}  UFU(P_i)$, 
y se toma $s=F(\beta)$.
Así, se tiene el siguiente diagrama en $\mathcal{OP}$.

\begin{equation}\label{dg-coeq-refls}
\begin{gathered}
\xymatrix@C=4pc{
F \left( \underset{i\in I}{\text{colim}\,}  UFU(P_i) \right) \ar@<0.5ex>[r]^{d_1} \ar@<-0.5ex>[r]_{d_0} &
F(\underset{i\in I}{\text{colim}\,}  U(P_i)) \ar@/_2pc/[l]_-{s}
}
\end{gathered}
\end{equation}

Antes de toamr el coproducto reflexico de este diagrama, se debe de verificar que $d_1s=d_0s=1$.
Para mostrar esto solo hay que verificar que sus componentes definidas sobre $U(P_i)$,
$d_1s$ y $d_0s$, son ambas iguales a la identidad.

El morfismo $s$ es determinado por $\eta_{U(P_i)}:U(P_i)\to UFU(P_i)$,
$d_0$ por $UF(\alpha_i):UFU(P_i)\to UF(\underset{i\in I}{\text{colim}\,}  U(P_i))$. 
La naturalidad de $\eta$ hace el siguiente diagrama conmutativo.

\begin{equation}\label{dg-eta-nat-up}
\begin{gathered}
\xymatrix{
U(P_i) \ar[r]^-{\eta_{U(P_i)}} \ar[d]_-{\alpha_i}& UFU(P_i) \ar[d]^-{UF(\alpha_i)} \\
\underset{i\in I}{\text{colim}\,}  U(P_i) \ar[r]^-{\eta} & UF(\underset{i\in I}{\text{colim}\,}  U(P_i))
}
\end{gathered}
\end{equation}

Así, se tiene que 
$UF(\alpha_i)\eta_{U(P_i)}=\eta \alpha_i:U(P_i)\to UF(\underset{i\in I}{\text{colim}\,}  U(P_i))$,
lo cual induce la identidad sobre
$UF(\underset{i\in I}{\text{colim}\,}  U(P_i))$.
Para $d_1s$, $d_1$ es determinado por la composición \ref{dg-d1-def}, 
entonces se tiene que $d_1s$ es determinado por la composición
$\eta \alpha_i U(\epsilon_{P_i})\eta_{U(P_i)}$, which by the triangular equations
de la unidad y la counidad (proposición \ref{prop-triangular-eq-unit-counit}),
es igual a $\eta \alpha_i$, 
lo cual como antes, induce la identidad sobre $UF(\underset{i\in I}{\text{colim}\,}  U(P_i))$.
Entonces por proposición \ref{prop-coeq-reflex-OP} existe el coecualizador del diagrama \ref{dg-coeq-refls},

\begin{equation}
\begin{gathered}
\xymatrix@C=4pc{
F \left( \underset{i\in I}{\text{colim}\,}  UFU(P_i) \right) \ar@<0.5ex>[r]^{d_1} \ar@<-0.5ex>[r]_{d_0} &
F(\underset{i\in I}{\text{colim}\,}  U(P_i)) \ar@/_2pc/[l]_-{s} \ar[r]^-q& Q
}
\end{gathered}
\end{equation}

El operad $Q$ será el colimite de la colección $\{P_i\}_{i\in I}$.
Solo se debe de verificar la existencia de los morfismos de operads desde cada $P_i$ a $Q$
y la propiedad universal para los colímites.
Para hacer eso, primero vamos a ver la información que da el ser un coecualizador de $d_0$ y $d_1$.

Considere $R$ un operad y un morfismo de operads $f:F(\underset{i\in I}{\text{colim}\,}  U(P_i))\to R$ such that
$fd_0=fd_1$. Este morfismo está determinado por sus componentes 
$h_i:U(P_i)\to U(R)$ dadas por las composiciones,

\begin{equation}
\begin{gathered}
\xymatrix{
U(P_i)\ar[r]^-{\alpha_i} \ar@/_2pc/[rr]^-{h_i}& \underset{i\in I}{\text{colim}\,}  U(P_i) \ar[r]^-{\theta(f)} & U(R)
}
\end{gathered}
\end{equation}

Los morfismos $fd_0$ y $fd_1$ son determinados por los morfismos de $UFU(P_i)$ a $U(R)$ en $\mathbb{S}$-Mod,
entonces se describirá en términos de sus componentes la relación $fd_1=fd_0$.
En el caso de $fd_0$ recuerde que $d_0$ está determinado por los morfismos
$UF(\alpha_i):UFU(P_i)\to UF(\underset{i\in I}{\text{colim}\,}  U(P_i))$ 
y considere el siguiente diagrama conmutativo.

\begin{equation}\label{dg-big-triangle}
\begin{gathered}
\xymatrix{
UFU(P_i) \ar[rd]^-{UF(\alpha_i)} \ar@/^4pc/@{-->}[rrrddd]^-{U(\theta^{-1}(h_i))} & &&\\
&   UF(\underset{i\in I}{\text{colim}\,}  U(P_i)) \ar[rrdd]^-{U(f)}&& \\
 & \underset{i\in I}{\text{colim}\,} U(P_i) \ar[u]^-{\eta}\ar[rrd]_-{\theta(f)} &  &\\
U(P_i) \ar[uuu]^-{\eta_{U(P_i)}} \ar[ur]^-{\alpha_i} \ar[rrr]_-{h_i}& & & U(R)
}
\end{gathered}
\end{equation}

El cuadrilátero es conmutativo por la naturalidad de la counidad. 
Entonces  la composición en la diagonal $U(f)UF(\alpha_i)$ 
está determinada por la última parte, es decir $h_i$, 
y la biyección de la adjunción
$F\vdash U$ dice que $U(f)UF(\alpha_i)$ es igual a $U(\theta^{-1}(h_i))$,
lo cual significa que $fd_0$ está determinado por $U(\theta^{-1}(h_i)):UFU(P_i)\to U(R)$.

Para $fd_1$, recuerde que $d_1$ es determinado por la composición \ref{dg-d1-def}, entonces $fd_1$
es determinado por la composición,

\begin{equation}
\begin{gathered}
\xymatrix{
UFU(P_i) \ar[r]^{U(\epsilon_{P_i})} & U(P_i) \ar[r]^-{\alpha_i} & \underset{i\in I}{\text{colim}\,}  
U(P_i) \ar[r]^-{\eta} & UF(\underset{i\in I}{\text{colim}\,}  U(P_i))
\ar[r]^-{U(f)} & U(R)
}
\end{gathered}
\end{equation}

Por \ref{dg-big-triangle} se tiene que $U(f)\eta \alpha_i=\theta(f)\alpha_i=h_i$, 
entonces $fd_1$ está determinado por la composición
$U(\epsilon_{P_i})h_i:UFU(P_i)\to U(R)$. 
Junto con el resultado para $fd_0$, dice que $fd_1=fd_0$ si y solo si el
siguiente diagrama es conmutativo.

\begin{equation}\label{dg-ope-morphs-fi}
\begin{gathered}
\xymatrix{
UFU(P_i)\ar[d]_-{U(\epsilon_{P_i})}  \ar[rd]^-{U(\theta^{-1}(h_i))} &
  \\
U(P_i) \ar[r]_-{h_i} & U(R)
}
\end{gathered}
\end{equation}

Este diagrama es conmutativo si y solo si $h_i$ 
es un morfismo de operads, en otras palabras,
si existe un morfismo de operads $f_i:P_i\to R$ tal que $U(f_i)=h_i$. 

Suponga que \ref{dg-ope-morphs-fi} es conmutativo. Se necesita probar que $h_i$
preserva la estructura operádica en $P_i$, 
es decir, se necesita verificar las condiciones de la definición \ref{df-operads-morphism}. 
Para evitar confusiones, se denota $\lambda$ la unidad del operad en esta parte.

\begin{enumerate}
\item La unidad. 

\begin{align}
h_i U(\lambda_{P_i}) &= h_i U(\epsilon_{P_i}\lambda_{FU(P_i)})\\
&= h_i U(\epsilon_{P_i})U(\lambda_{FU(P_i)})\\
&=U(\theta^{-1}(h_i))U(\lambda_{FU(P_i)})\\
&=U(\theta^{-1}(h_i)\lambda_{FU(P_i)})=U(\lambda_{R})\\
\end{align}

\item Equivarianza es consecuencia del hecho que todos son morfismos de $\mathbb{S}$-modules.

\item La composición.

\begin{align}
h_i U(\gamma_{P_i}) 
&= h_i U(\gamma_{P_i}1_{P_i}) \\
&=h_i U(\gamma_{P_i})1_{U(P_i)} \\
&=h_i U(\gamma_{P_i}) U(\epsilon_{P_i})\eta_{U(P_i)}\\
&=h_i U(\gamma_{P_i}\epsilon_{P_i})\eta_{U(P_i)}\\
&=h_i U(\epsilon_{P_i}\gamma_{FU(P_i)})\eta_{U(P_i)}\\
&=h_i U(\epsilon_{P_i})U(\gamma_{FU(P_i)})\eta_{U(P_i)}\\
&=U(\theta^{-1}(h_i))U(\gamma_{FU(P_i)})\eta_{U(P_i)}\\
&=U(\theta^{-1}(h_i)\gamma_{FU(P_i)})\eta_{U(P_i)}\\
&=U(\gamma_{R}\theta^{-1}(h_i))\eta_{U(P_i)}\\
&=U(\gamma_{R})U(\theta^{-1}(h_i))\eta_{U(P_i)}\\
&=U(\gamma_{R})h_i
\end{align}
\end{enumerate}

Recíprocamente, suponga que existe una colección de morfismos de operads
$\{f_i:P_i\to R\}_{i\in I}$, 
tal que $U(f_i)=h_i$ para cada $i\in I$.  
Entonces el siguiente triángulo es conmutativo por la naturalidad de $\theta^{-1}$.

\begin{equation}
\begin{gathered}
\xymatrix{
FU(P_i)\ar[rd]^-{\theta^{-1}(U(f_i))} \ar[d]_-{\epsilon_{P_i}}&\\
P_i \ar[r]_-{f_i}&R
}
\xymatrix{
&\mathcal{OP}(FU(P_i),R) & \mathbb{S}\text{-Mod}(U(P_i),U(R)) \ar[l]_-{\theta^{-1}}\\
&\mathcal{OP}(FU(P_i),P_i) \ar[u]^-{{f_i}_*} & \mathbb{S}\text{-Mod}(U(P_i),U(P_i))\ar[u]_-{U(f_i)_*} \ar[l]_-{\theta^{-1}}
}
\end{gathered}
\end{equation}

Entonces el diagrama \ref{dg-ope-morphs-fi} es conmutativo. 
Ahora se pasa a verificar que
$Q$ es el colímite de la colección de operads $\{P_i\}_{i\in I}$.
Se vió que esta colección induce morfismos $h_i:U(P_i)\to U(Q)$ de $\mathbb{S}$-modules
que satisfacen \ref{dg-ope-morphs-fi}, 
entonces definen morfismos de operads $f_i:P_i\to Q$,
tal que $U(f_i)=h_i$. Estos morfismos son las aristas del cocono.

Cualquier colección de morfismos de operads $f_i:P_i\to R$, 
define un morfismo de operads
$f$ de $F(\underset{i\in I}{\text{colim}\,} U(P_i))$ a $R$ tal que $fd_0=fd_1$. 
Entonces existe un único morfismo de operads $g:Q\to R$ tal que $gq=f$.
El morfismo $g$ conmuta con las aristas del cocono, y esto exhibe a $Q$
como el colímite de $\{P_i\}_{i\in I}$.
\end{proof}

\section{Agradecimientos}

El autor agradece el financiamiento de la Vicerrectoría de Investigación de la Universidad de Costa Rica a través del proyecto \textbf{821-B6-A19}.

\bibliographystyle{amsplain}
\bibliography{principal-bibliography}
\end{document}